\pdfoutput=1
\RequirePackage{ifpdf}
\ifpdf 
\documentclass[pdftex]{sigma}
\else
\documentclass{sigma}
\fi

\newcommand{\campoA}{\textbf{A}}
\newcommand{\campoX}{\textbf{X}}

\newcommand{\campoY}{\textbf{Y}}
\newcommand{\campov}{\textbf{v}}

\newcommand{\intprod}{\mathbin{\raisebox{\depth}{\scalebox{1}[-1]{$\lnot$}}}}

\numberwithin{equation}{section}

\newtheorem{Theorem}{Theorem}[section]
\newtheorem{Lemma}[Theorem]{Lemma}
\newtheorem{Proposition}[Theorem]{Proposition}
 { \theoremstyle{definition}
\newtheorem{Definition}[Theorem]{Definition}
\newtheorem{Remark}[Theorem]{Remark} }

\begin{document}


\newcommand{\arXivNumber}{1511.03025}

\renewcommand{\PaperNumber}{077}

\FirstPageHeading

\ShortArticleName{Solvable Structures Associated to the Nonsolvable Symmetry Algebra $\mathfrak{sl}(2,\mathbb{R})$}

\ArticleName{Solvable Structures Associated\\ to the Nonsolvable Symmetry Algebra $\boldsymbol{\mathfrak{sl}(2,\mathbb{R})}$}

\Author{Adri\'an RUIZ and Concepci\'on MURIEL}

\AuthorNameForHeading{A.~Ruiz and C.~Muriel}

\Address{Department of Mathematics, University of C\'adiz, 11510 Puerto Real, Spain}
\Email{\href{mailto:adriruse@hotmail.com}{adriruse@hotmail.com}, \href{mailto:concepcion.muriel@uca.es}{concepcion.muriel@uca.es}}

\ArticleDates{Received November 11, 2015, in f\/inal form August 03, 2016; Published online August 08, 2016}

\Abstract{Third-order ordinary dif\/ferential equations with Lie symmetry algebras isomorphic to the nonsolvable algebra $\mathfrak{sl}(2,\mathbb{R})$ admit solvable structures. These solvable structures can be constructed by using the basis elements of these algebras. Once the solvable structures are known, the given equation can be integrated by quadratures as in the case of solvable symmetry algebras.}


\Classification{34A05; 34A26; 34C14} 

\vspace{-2.5mm}

\section{Introduction}

It is well known that an $n$th-order ordinary dif\/ferential equation (ODE) possessing an $n$-di\-men\-sional solvable Lie symmetry algebra can be integrated by quadratures \cite{hydon2000symmetry,ibragimov2004practical,olver2000applications, olverequivalence, hans1989differential}. This is a suf\/f\/icient condition, but not necessary \cite{ artemioecu} and extensions of the classical approach have been amply investigated in the recent literature (hidden symmetries \cite{shrauner95,hidden1993}, nonlocal symmetries \cite{adam2002,govinder95,govinder97}, $\mathcal{C}^{\infty}$-symmetries \cite{muriel01ima1}, etc.).

In this paper we focus on a generalization of solvable algebras called solvable structures \cite{barko,basarab,athorne,prince} and their applications to integrate ODEs which admit nonsolvable symmetry algebras. The concept of solvable structure \cite{basarab} refers to systems of independent vector f\/ields that are in involution; for a scalar ODE, this (trivially involutive) system is formed by just one element, the vector f\/ield $\campoA$ associated to the given ODE. In this case, a solvable structure involves an ordered set of generalized vector f\/ields that, in general, are not symmetries of the ODE and do not form a solvable algebra. Nevertheless, if a solvable structure for the ODE is known, then the equation can be (locally) solved by quadratures \cite{basarab,athorne}. Thus, it is important to have methods that allow the determination of solvable structures for ODEs in order to solve them by quadratures.

We investigate in this paper if a solvable structure can be constructed from a nonsolvable Lie symmetry algebra. In this case, the integrability by quadratures of the ODEs admitting nonsolvable symmetry algebras would be warranted, as in the case of solvable symmetry algebras. This study addresses the case of third-order ODEs which admit a Lie symmetry algebra isomorphic to the nonsolvable algebra $\mathfrak{sl}(2,\mathbb{R})$. This class of equations is well determined and known in the literature: in fact, symmetry analysis of third-order ODEs and the classif\/ication according to their symmetries have been extensively studied \cite{gatomri,Gusyatnikova,mahomed,schmucker1998symmetry}.

Previous studies in the literature \cite{olverclarkson, ibragimovnucci} show that any third-order ODE admitting ${\rm SL}(2,\mathbb{C})$ as symmetry group can be solved via a pair of quadratures and the solution to a~Riccati equation, or, equivalently, a second-order linear ODE. The analysis performed in \cite{olverclarkson} proves this result by using that the three inequivalent actions of the unimodular Lie group ${\rm SL}(2,\mathbb{C})$ on the complex plane \cite{lie} are directly connected via the standard prolongation process. A~dif\/ferent approach is presented in \cite{ibragimovnucci}: a two-dimensional subalgebra $\mathcal{L}_2$ is used to reduce the given third-order ODE to a~f\/irst-order equation, which cannot be integrated by quadratures, but can be transformed into a~Riccati equation by using a nonlocal symmetry. In~\cite{algebras_no_resolubles} it is shown that this nonlocal symmetry comes from one of the symmetries of the original third-order ODE and can be recovered as a~$\mathcal{C}^{\infty}$-symmetry for the f\/irst-order equation. The general solution of the Riccati equation becomes a~second-order ODE in the original variables, which can be integrated by quadratures by using~$\mathcal{L}_2$. Both procedures require a previous transformation to map the basis elements of the Lie symmetry algebra into one of the canonical realizations of $\mathfrak{sl}(2,\mathbb{C})$ \cite{accion,lie}.

The starting point of the approach presented in this paper is a basis $\left\{ \campov_1, \campov_2, \campov_3 \right\}$ of the Lie invariance algebra for the given third-order ODE satisfying the following commutation relations:
\begin{gather*}
[\campov_1,\campov_3]=\campov_1,\qquad [\campov_1,\campov_2]=2 \campov_3,\qquad [\campov_3,\campov_2]=\campov_2.
\end{gather*}
Although this symmetry Lie algebra is isomorphic to the nonsolvable symmetry algebra $\mathfrak{sl}(2,\mathbb{R})$,
it is shown that solvable structures can be explicitly constructed from the second-order prolongations of the basis elements (denoted by $\campov_i^{(2)}$ for $i=1,2,3$). In fact, the main result in this paper shows that there exist two functions $F_1$ and $F_2$ such that the vector f\/ields $F_1\campov_1^{(2)}$, $F_2\campov_2^{(2)}$, and $\campov_3^{(2)}$ can be used to construct two solvable structures with respect to the vector f\/ield $\campoA$ associated to the given ODE (Theorem~\ref{teorema_final}). As a conclusion, the original ODE can be solved (at least locally) by quadratures, provided the functions~$F_1$ and~$F_2$ are known.

This result is based on the compatibility of two systems of partial dif\/ferential equations (PDEs) (see systems~(\ref{tres})). We present a constructive proof of the existence of solutions for systems (\ref{tres}), which provides a~method to construct the solvable structure in practice. For this purpose, instead of reducing the original ODE to a f\/irst-order ODE in the usual way (by using the solvable subalgebra generated by $\campov_1$ and $\campov_2$, as in~\cite{olverclarkson,ibragimovnucci}), we use $\mathbf{v}_3$ to reduce the original equation to a second-order ODE. The basis elements $\campov_1$ and $\campov_2$ are lost as Lie point symmetries for the reduced equation, but can be recovered as $\mathcal{C}^{\infty}$-symmetries~\cite{algebras_no_resolubles}. These $\mathcal{C}^{\infty}$-symmetries can be used to calculate two functionally independent f\/irst integrals $I_1$ and $I_2$ of the reduced equation. A~procedure to compute them by quadratures~\cite{primer_articulo} is explained in Section \ref{section2}, provided that particular solutions of systems~(\ref{hs}) are known. Particular solutions $F_1$ and $F_2$ of sys\-tems~(\ref{tres}) can be directly found by using $I_1$ and $I_2$ (see~(\ref{F})) or by using solutions of systems~(\ref{hs}) (see~(\ref{Fconhbarra})).

In Section~\ref{section4} the solvable structures are used to give three dif\/ferent strategies that can be followed to integrate by quadratures any third-order ODE admitting a Lie symmetry algebra that is isomorphic to $\mathfrak{sl}(2,\mathbb{R})$. Remarkably, no further integration is necessary when the mentioned functions $F_1$ and $F_2$ come from f\/irst integrals~$I_1$ and~$I_2$ of the reduced equation or from particular solutions of systems~(\ref{hs}).

The method is illustrated in Sections~\ref{section6} and~\ref{section7} with two examples of third-order ODEs which admit the nonsolvable Lie algebra $\mathfrak{sl}(2,\mathbb{R})$ as symmetry algebra. For each example, the procedure provides a complete set of f\/irst integrals for the third-order equation in terms of two independent solutions of a second-order linear equation. As in the previous methods in the literature, the complete solution of the original ODE involves the solutions of a second-order linear ODE, but the use of the procedure we present in this paper (see Method~2 in Section~\ref{section4}) does not require additional quadratures.

\section{Symmetries and solvable structures}\label{preliminares}

In this section we recall the notion of solvable structure and some of its properties~\cite{basarab, prince}. Functions, vector f\/ields, and dif\/ferential forms are assumed to be smooth and well def\/ined on an open and simply connected subset~$D$ of either~$\mathbb{R}^n$ or an $n$-dimensional manifold $\mathcal{M}_n$.

Given a system $\mathcal{S}$ of vector f\/ields on $D$, $\operatorname{span}(\mathcal{S})$ stands for the space of the linear combinations of the elements of $\mathcal{S}$ over the ring of the smooth functions on $D$. In what follows, two systems of vector f\/ields, $\mathcal{S}$ and $\mathcal{\widetilde{S}}$, are called \textit{equivalent} if $\operatorname{span}(\mathcal{S})= \operatorname{span}(\mathcal{\widetilde{S}})$ \cite{athorne}.

The following concepts refer to an involutive system of vector f\/ields, i.e., a set of vector f\/ields $\mathcal{A}=\{ \campoA_1,\dots, \campoA_{r} \}$ such that $[\campoA_i,\campoA_j]\in \operatorname{span}(\mathcal{A})$ for $i,j \in \{1,\dots,r\}$. In this paper angle brackets are used to indicate that the order of the elements of the set must be taken into account.

\begin{Definition}\label{simetria}
 {\rm Let $\mathcal{A}= \{ \campoA_1,\dots, \campoA_{r} \}$ be a system of $r<n$ independent vector f\/ields on $D$ which are in involution.
 \begin{enumerate}\itemsep=0pt
 \item A smooth vector f\/ield $\campoX$ on $D$ is called a \textit{symmetry} of $\mathcal{A}$ if the following conditions hold:
 \begin{enumerate}\itemsep=0pt
 \item $\campoA_1,\dots, \campoA_{r}$, and $\campoX$ are independent;
 \item $[\campoX,\campoA_i]\in \operatorname{span}(\mathcal{A})$ for $1\leq i\leq r$.
 \end{enumerate}
 \item Let $\mathcal{S}= \langle \campoX_1,\dots,\campoX_{n-r} \rangle$ be an ordered set of independent vector f\/ields on~$D$. The ordered system $\mathcal{A} \cup \mathcal{S} = \langle \campoA_1,\dots, \campoA_{r}, \campoX_1,\dots,\campoX_{n-r} \rangle$ is a \textit{solvable structure} with respect to~$\mathcal{A}$ if
 \begin{enumerate}\itemsep=0pt
 \item $\mathcal{S}_j= \{ \campoA_1,\dots, \campoA_{r}, \campoX_1,\dots,\campoX_j\}$ is in involution for $j=1,\dots,n-r$;
 \item $\campoX_1$ is a symmetry of $\mathcal{A}$;
 \item $\campoX_{j+1}$ is a symmetry of $\mathcal{S}_j$ for $j=1,\dots,n-r-1$.
 \end{enumerate}
 \end{enumerate}}
\end{Definition}

\subsection{Symmetries and solvable structures in the context of ODEs}

The notion of symmetry given in Def\/inition \ref{simetria} represents a generalization of the concept of Lie point symmetry. Given a scalar $n$th-order ODE
\begin{gather}\label{ode}
u_n = \phi(x,u,u_1,\dots,u_{n-1}),
\end{gather} where $x$ denotes the independent variable, $u$ is the dependent variable, and $u_j=\frac{d^j u}{dx^j}$ for $1\leq j\leq n$, let \begin{gather}\label{campoA}
\campoA = \partial_x + u_1 \partial_u + \dots+ \phi (x,u,u_1,\dots,u_{n-1})\partial_{u_{n-1}}
\end{gather} denote \looseness=-1 the vector f\/ield associated to equation~(\ref{ode}). Equation (\ref{ode}) is def\/ined for points of the corresponding $n$th-order jet space whose projections to the $(n-1)$th-order jet belong to the domain of~$\phi$. Let $M\subset \mathbb{R}^2$ be an open set of the projection of this domain to the zero-order jet space.

A smooth vector f\/ield $\campov = \xi(x,u) \partial_x + \eta(x,u) \partial_u$ on $M$ is a Lie point symmetry of equation~(\ref{ode}) if and only if
\begin{gather*}\big[\campov^{(n-1)},\campoA\big]=-\campoA(\xi)\cdot \campoA,\end{gather*}
where $\campov^{(n-1)}$ stands for the $(n-1)$th-order {prolongation of} $\campov$ \cite{ibragimov2004practical, olver2000applications, olverequivalence, hans1989differential}. Therefore, $\campov^{(n-1)}$ is a symmetry of the (trivially involutive) system $\mathcal{A}=\{\campoA\}$ in the sense of Def\/inition \ref{simetria}. The same result holds for generalized symmetries for which the inf\/initesimals $\xi$ and $\eta$ can depend on derivatives of $u$ with respect to $x$ \cite{olver2000applications, olverequivalence}.

The notion of solvable structure given in Def\/inition \ref{simetria} generalizes the concept of solvable symmetry algebra. If equation (\ref{ode}) admits a solvable symmetry algebra $\mathcal{G}$ of dimension $n$, then there exists an ordered basis $\langle\campov_1,\dots,\campov_n\rangle$ of $\mathcal{G}$ such that $[\campov_i,\campov_j]=\sum\limits_{k=1}^{j-1}c_{ij}^k\campov_k$ for $1\leq i<j\leq n$ and where $c_{ij}^k\in\mathbb{R}$. Therefore, $\big\langle\campoA,\campov_1^{(n-1)},\dots,\campov_n^{(n-1)} \big\rangle$ is a solvable structure with respect to~$\{\campoA\}$.

The integrability by quadratures of an $n$th-order ODE which admits a solvable symmetry algebra $\mathcal{G}$ of dimension~$n$ is well known. In fact, the integrability by quadratures can be cha\-rac\-terized through solvable structures:

\begin{Proposition}[\protect{\cite[Proposition~6]{basarab}}] \label{basarab}
An involutive system $\mathcal{A}$ is locally integrable by quadratures if and only if there exists a solvable structure with respect to~$\mathcal{A}$.
\end{Proposition}

Now we recall the method \cite{basarab,athorne} to construct $n$ independent f\/irst integrals for equation (\ref{ode}) when a solvable structure $\langle \campoA,\campoX_1, \dots , \campoX_n \rangle$ with respect to $\{\campoA\}$ is known. The solvable structure is used to def\/ine the dif\/ferential 1-forms given by
\begin{gather}\label{formas_b}
\boldsymbol{\omega}_i = \frac{\campoX_n \intprod \cdots \intprod \widehat{\campoX_{i}} \intprod \cdots \intprod \campoX_1 \intprod \campoA \intprod \boldsymbol{\Omega}}{\campoX_n \intprod \cdots \intprod \campoX_1 \intprod \campoA \intprod \boldsymbol{\Omega}}, \qquad i=1,\dots,n,
\end{gather}
where $\widehat{\campoX_i}$ indicates omission of $\campoX_i$, $\intprod$ denotes the interior product, and $\boldsymbol{\Omega} = dx \wedge du\wedge \dots \wedge du_{n-1}$. The system $\{\boldsymbol{\omega}_1,\dots,\boldsymbol{\omega}_n\}$ has distinguishing closure properties \cite{athorne}: $d\boldsymbol{\omega}_n=0$ and for $1\leq i< n$, $d\boldsymbol{\omega}_i\in \mathcal{I}\{\boldsymbol{\omega}_{i+1},\dots,\boldsymbol{\omega}_{n}\}$,
where $\mathcal{I}\{\boldsymbol{\omega}_{i+1},\dots,\boldsymbol{\omega}_{n}\}$ denotes the ideal generated by $\boldsymbol{\omega}_{i+1},\dots,\boldsymbol{\omega}_{n}$ under taking exterior products.

These properties permit the integration by quadratures (at least locally) of the 1-forms (\ref{formas_b}) by proceeding as follows: $\boldsymbol{\omega}_n$ is locally exact and any of its primitives $I_n$ is a f\/irst integral of~$\campoA$. The restriction of $\boldsymbol{\omega}_{n-1}$ to each submanifold def\/ined by $I_n=c_n$, $c_n\in \mathbb{R}$, is closed, and a primitive $I_{n-1}$ can be found by a quadrature. We can continue in this fashion by further restricting the submanifolds at each stage until we have fully integrated the system of 1-forms. By the def\/inition of the 1-forms~(\ref{formas_b}), the functions $\{I_1,\dots,I_n\}$ are functionally independent f\/irst integrals of $\campoA$. These results provide the following theorem~\cite{basarab,athorne}:

\begin{Theorem} \label{formas_exactas}
Let \eqref{campoA} be the vector field associated to equation \eqref{ode} and assume that $\langle \mathbf{A},\mathbf{X}_1,$ $\dots , \mathbf{X}_n \rangle$ is a solvable structure with respect to $\{\mathbf{A}\}$. Then the given ODE \eqref{ode} can be $($at least locally$)$ solved by quadratures alone.
\end{Theorem}

\section[Symmetries and $\mathcal{C}^{\infty}$-symmetries for second-order ODEs]{Symmetries and $\boldsymbol{\mathcal{C}^{\infty}}$-symmetries for second-order ODEs}\label{section2}

In this section we establish some relationships between symmetries and $\mathcal{C}^{\infty}$-symmetries (also called $\lambda$-symmetries) for the integrability by quadratures of second-order ODEs that will be used later.

We consider a second-order equation
\begin{gather}\label{ode2}
w_{2}=\widetilde{\phi}(y,w,w_1),
\end{gather}
where $y$ is the independent variable, $w$ is the dependent variable, and $w_i=\frac{d^i w}{dy^i}$ for $i\in {\mathbb N}$. Let $M_1\subset \mathbb{R}^2$ be an open set of the projection of the domain of $\widetilde{\phi}$ to the corresponding zero-order jet space. Throughout this section~$\campoA$ denotes the vector f\/ield associated to (\ref{ode2}) and $\mathbf{D}_y$ is the total derivative operator with respect to $y$, i.e., $\mathbf{D}_y=\partial_y+w_1\partial_w+\dots+w_k\partial_{w_{k+1}} +\cdots$.

We recall \cite{muriel01ima1} that a $\mathcal{C}^{\infty}$-symmetry of (\ref{ode2}) is a pair $(\campov,\lambda)$, where $
\campov = \xi(y,w) \partial_y + \eta(y,w) \partial_w$ is a vector f\/ield on $M_1$ and $\lambda= \lambda(y,w,w_1)$ is a smooth function, such that
\begin{gather}\label{corchetelambda}
\big[\campov^{[\lambda,(1)]},\campoA\big]=\lambda\campov^{[\lambda,(1)]}-(\campoA+\lambda)(\xi)\campoA,
\end{gather} where $\campov^{[\lambda,(1)]}$ stands for the f\/irst-order $\lambda$-prolongation of~$\campov$
\begin{gather}\label{lambdacoor}
\campov^{[\lambda,(1)]}=\xi\partial_y + \eta \partial_w+\left((\mathbf{D}_y+\lambda)(\eta)-(\mathbf{D}_y+\lambda)(\xi)w_1\right)\partial_{w_1}.
\end{gather}

Two $\mathcal{C}^{\infty}$-symmetries $(\overline{\mathbf{v}}_1,\lambda_1)$ and $(\overline{\mathbf{v}}_2,\lambda_2)$ of equation (\ref{ode2}) are called $\campoA$-equivalent \cite{muriel2009} (or simply equivalent) if the systems $\big\{\campoA,\overline{\mathbf{v}}_1^{[\lambda_1,(1)]}\big\}$ and $\big\{\campoA,\overline{\mathbf{v}}_2^{[\lambda_2,(1)]}\big\}$ are equivalent, i.e.,
\begin{gather*}
\operatorname{span}\bigl(\big\{\campoA,\overline{\mathbf{v}}_1^{[\lambda_1,(1)]}\big\}\bigr)
=\operatorname{span}\bigl(\big\{\campoA,\overline{\mathbf{v}}_2^{[\lambda_2,(1)]}\big\}\bigr).\end{gather*}
We assume that $(\overline{\mathbf{v}}_1,\lambda_1)$ and $(\overline{\mathbf{v}}_2,\lambda_2)$ are two inequivalent $\mathcal{C}^{\infty}$-symmetries of the equation~(\ref{ode2}) and denote $\campoY_i=\overline{\mathbf{v}}_i^{[\lambda_i,(1)]}$ for $i=1,2$. The corresponding expressions~(\ref{corchetelambda}) become
\begin{gather}\label{corcheteAY}
[\campoY_i,\campoA]=\lambda_i\campoY_i+\mu_i\campoA,
\end{gather} where $\mu_i=-(\campoA+\lambda_i)(\campoY_i(y))$ for $i=1,2$. Assuming that $\lambda_1\neq0$ and $\lambda_2\neq0$, (\ref{corcheteAY}) shows that~$\campoY_1$ and~$\campoY_2$ are not symmetries of $\{\campoA\}$ in the sense of Def\/inition \ref{simetria}. Relations (\ref{corcheteAY}) imply that the systems $\{\campoA, \campoY_1\}$ and $\{\campoA, \campoY_2\}$ are in involution; by Frobenius theorem~\cite{olver2000applications}, there exist two functionally independent f\/irst integrals $I_1=I_1(y,w,w_1)$ and $I_2=I_2(y,w,w_1)$ of $\campoA$ such that $\campoY_1(I_1)=\campoY_2(I_2)=0$. Such nonconstant function $I_1$ (resp.~$I_2$) will be called a f\/irst integral of~$\campoA$ associated to the $\mathcal{C}^{\infty}$-symmetry $(\overline{\mathbf{v}}_1,\lambda_1)$ (resp.~$(\overline{\mathbf{v}}_2,\lambda_2)$).

A procedure to calculate by quadratures two f\/irst integrals associated to two inequivalent $\mathcal{C}^{\infty}$-symmetries of a~second-order ODE is described in the following subsection.

\subsection[$\mathcal{C}^{\infty}$-symmetries and integrability by quadratures for second-order ODEs]{$\boldsymbol{\mathcal{C}^{\infty}}$-symmetries and integrability by quadratures for second-order ODEs}\label{subsection1}

 Any given $\mathcal{C}^{\infty}$-symmetry $(\campov,\lambda)$ of equation (\ref{ode2}) is equivalent to the $\mathcal{C}^{\infty}$-symmetry \begin{gather} \label{canonico}
 (\partial_w,\lambda_Q),\qquad \text{where}\quad \lambda_Q=\lambda+ \frac{\campoA(Q)}{Q},
 \end{gather} and $Q$ denotes the characteristic $Q=\eta-\xi\cdot w_1$ of $\campov$. This is a consequence of the relation
 \begin{gather} \label{enum_proop}
 \campov^{[\lambda,(1)]}=Q(\partial_w)^{[\lambda_{Q},(1)]}+\xi\campoA,
 \end{gather} which follows from (\ref{lambdacoor}). This pair $(\partial_w,\lambda_Q)$ is called the \textit{canonical representative} of $(\campov,\lambda)$.

If $(\overline{\mathbf{v}}_1,\lambda_1)$ and $(\overline{\mathbf{v}}_2,\lambda_2)$ are two inequivalent $\mathcal{C}^{\infty}$-symmetries of equation~(\ref{ode2}) and~$Q_1$ and~$Q_2$ are the respective characteristics, then the corresponding expressions (\ref{enum_proop}) can be written as
 \begin{gather}\label{clYX}
 \campoY_1=Q_1\campoX_1+\xi_1\campoA, \qquad \campoY_2=Q_2\campoX_2+\xi_2\campoA,
 \end{gather} where $\campoX_i=(\partial_w)^{[\lambda_{Q_i},(1)]}$ and $\campoY_i=\overline{\mathbf{v}}_i^{[\lambda_i,(1)]}$ for $i=1,2$.

Although, in general, $\campoY_1$ and $\campoY_2$ are not in involution, the vector f\/ields $\campoX_1$ and $\campoX_2$ form a~two-dimensional algebra. In fact, it can be checked that
\begin{gather} \label{lambdascor}
 [\campoX_1, \campoA]= \lambda_{Q_1} \campoX_1, \qquad [\campoX_2, \campoA]= \lambda_{Q_2}\campoX_2, \qquad [\campoX_1,\campoX_2]=\rho(\campoX_1 -\campoX_2),
 \end{gather}
 where
 \begin{gather*}\rho = \frac{\campoX_1(\lambda_{Q_2})- \campoX_2(\lambda_{Q_1})}{\lambda_{Q_1}-\lambda_{Q_2}}.\end{gather*}

By using (\ref{lambdascor}) and the properties of the Lie bracket, it can be proved that if $\overline{h}_1, \overline{h}_2 \in \mathcal{C^{\infty}}(M_1^{(1)})$ satisfy the following systems
 \begin{alignat}{3}
& \campoA(\overline{h}_1) = \lambda_{Q_1} \overline{h}_1,\qquad &&\campoX_2(\overline{h}_1) = \rho \overline{h}_1;&\nonumber\\
& \campoA(\overline{h}_2) = \lambda_{Q_2} \overline{h}_2,\qquad &&\campoX_1(\overline{h}_2) = \rho \overline{h}_2, &\label{hs}
 \end{alignat}
 then $\{\campoA, \overline{h}_1 \campoX_1, \overline{h}_2 \campoX_2 \}$ is an abelian algebra; hence the set $\{\overline{h}_1 \campoX_1, \overline{h}_2 \campoX_2\}$ is a system of commuting symmetries of $\{\campoA\}$. The compatibility of systems~(\ref{hs}) has been proved in~\cite{primer_articulo}.

If  $\overline{h}_1$ and $\overline{h}_2$ are some known particular solutions of  the respective system in~(\ref{hs}) then equation~(\ref{ode2}) can be integrated by quadratures: by using Theorem~\ref{formas_exactas} it can be checked that the dif\/ferential 1-forms
 \begin{gather} \label{omegas0}
 \boldsymbol{\beta}_1={\mu_1}({\campoX_1\intprod\campoA\intprod\boldsymbol{\Omega}}), \qquad \boldsymbol{\beta}_2={\mu_2}({\campoX_2\intprod\campoA\intprod\boldsymbol{\Omega}}),
 \end{gather} where
 \begin{gather}
 \mu_1=\frac{1}{\overline{h}_2\left({ \campoX_2\intprod\campoX_1\intprod\campoA\intprod\boldsymbol{\Omega}}\right)}=\frac{1}{\overline{h}_2 (\lambda_{Q_1} - \lambda_{Q_2})}, \nonumber\\
 \mu_2 = \frac{1}{\overline{h}_1\left({ \campoX_1\intprod\campoX_2\intprod\campoA\intprod\boldsymbol{\Omega}}\right)}=\frac{1}{\overline{h}_1 (\lambda_{Q_2} - \lambda_{Q_1})},\label{mu_1}
\end{gather}
 are (locally) exact. Let $I_i=I_i(y,w,w_1)$ be a function such that $\boldsymbol{\beta}_i=dI_i$ for $i=1,2$. By~(\ref{omegas0}), such functions $I_1$ and $I_2$ are f\/irst integrals of $\campoA$ and satisfy $ \campoX_1(I_1)=\campoX_2(I_2)=0$. According to~(\ref{clYX}), $I_1$ (resp.~$I_2$) is a common f\/irst integral for the involutive system $\{\campoA,\campoY_1\}$ (resp.~$\{\campoA,\campoY_2\}$).

\looseness=-1 Previous discussion shows that two functions $\overline{h}_1$ and $\overline{h}_2$ satisfying the corresponding system in~(\ref{hs}) can be used to construct, by quadratures, two f\/irst integrals, $I_1$ and $I_2$, of $\campoA$, associated to two inequivalent $\mathcal{C}^{\infty}$-symmetries of the equation. We want to point out that $\overline{h}_1$ and $\overline{h}_2$ are determined by the constructed functions $I_1$ and $I_2$, because by~(\ref{omegas0}), (\ref{mu_1}), and (\ref{clYX}), we can write
 \begin{gather} \label{hconX}
 \overline{h}_1=\frac{1}{\boldsymbol{\beta}_2(\campoX_1)}=\frac{1}{\campoX_1(I_2)}=\frac{Q_1}{\campoY_1(I_2)}, \qquad \overline{h}_2=\frac{1}{\boldsymbol{\beta}_1(\campoX_2)}=\frac{1}{\campoX_2(I_1)}=\frac{Q_2}{\campoY_2(I_1)}.
 \end{gather}
In fact, \looseness=-1 let $I_1$, $I_2$ be two arbitrary f\/irst integrals of $\campoA$ associated to $(\overline{\mathbf{v}}_1,\lambda_1)$ and $(\overline{\mathbf{v}}_2,\lambda_2)$, respectively; since $\campoA$, $\campoX_1$, and $\campoX_2$ are independent, then $\campoX_1(I_2)\neq 0$ and $\campoX_2(I_1)\neq 0$ and the functions
 \begin{gather}\label{hconX2}
 \overline{h}_1=\frac{1}{\campoX_1(I_2)}, \qquad \overline{h}_2=\frac{1}{\campoX_2(I_1)}
 \end{gather} are well def\/ined. From~(\ref{lambdascor}) it follows that $\campoA(\campoX_1(I_2))=-[\campoX_1,\campoA](I_2)=-\lambda_{Q_1}\campoX_1(I_2)$, because $\campoA(I_2)=0$.
 Therefore
 \begin{gather} \label{cuenta1}
 \campoA(\overline{h}_1)= \campoA\left(\frac{1}{\campoX_1(I_2)}\right)=
 -\frac{ \campoA\bigl(\campoX_1(I_2)\bigr)}{\bigl(\campoX_1(I_2)\bigr)^2}=
 \frac{\lambda_{Q_1}}{\campoX_1(I_2)}=\lambda_{Q_1}\overline{h}_1.
 \end{gather}
Equality \looseness=-1 $\campoA(\overline{h}_2)=\lambda_{Q_2}\overline{h}_2$ can be proved in a similar way.
 The third relation in (\ref{lambdascor}) provides $\campoX_1(\campoX_2(I_1))=[\campoX_1,\campoX_2](I_1)=-\rho\campoX_2(I_1)$ and $\campoX_2(\campoX_1(I_2))=-[\campoX_1,\campoX_2](I_2)=-\rho\campoX_1(I_2)$, because $\campoX_1(I_1)=\campoX_2(I_2)=0$. By proceeding as in (\ref{cuenta1}), it follows that $\campoX_1(\overline{h}_2)=\rho\overline{h}_2$ and $\campoX_2(\overline{h}_1)=\rho\overline{h}_1$. This proves that functions $\overline{h}_1$ and $\overline{h}_2$ in~(\ref{hconX2}) are solutions of the respective systems in~(\ref{hs}).

These results are collected in the following theorem for further reference.

\begin{Theorem}\label{integrabilidadorden2}
Let $(\overline{\mathbf{v}}_1,\lambda_1)$ and $(\overline{\mathbf{v}}_2,\lambda_2)$ be two inequivalent $\mathcal{C}^{\infty}$-symmetries of equation~\eqref{ode2} and consider their respective canonical representatives, $(\partial_w,\lambda_{Q_1})$ and $(\partial_w,\lambda_{Q_2})$. Denote $\mathbf{X}_i=(\partial_w)^{[\lambda_{Q_i},(1)]}$ and $\mathbf{Y}_i=\overline{\mathbf{v}}_i^{[\lambda_i,(1)]}$ for~$i=1,2$.
\begin{enumerate}\itemsep=0pt
\item[$1.$] If $\overline{h}_1$ and $\overline{h}_2$ are particular solutions of the respective systems in~\eqref{hs}, then two functionally independent first integrals of $\mathbf{A}$ associated to the given $\mathcal{C}^{\infty}$-sy\-mme\-tries can be found by quadratures as primitives of the $1$-forms defined in~\eqref{omegas0}.
\item[$2.$] Conversely, for $i=1,2$, let $I_i$ be a first integral of $\mathbf{A}$ associated to $(\overline{\mathbf{v}}_i,\lambda_i)$. Then the functions $\overline{h}_1$ and $\overline{h}_2$ given by~\eqref{hconX2} satisfy the corresponding systems in~\eqref{hs}.
 \end{enumerate}
\end{Theorem}

\section[Solvable structures from $\mathfrak{sl}(2,\mathbb{R})$ for third-order ODEs]{Solvable structures from $\boldsymbol{\mathfrak{sl}(2,\mathbb{R})}$ for third-order ODEs} \label{sec:math}

Let us consider a third-order ODE
\begin{gather}\label{ode3}
u_3 = \phi(x,u,u_1,u_2),
\end{gather}
that admits a Lie symmetry algebra that is isomorphic to $\mathfrak{sl}(2,\mathbb{R})$. In this section we investigate how a solvable structure for~(\ref{ode3}) can be explicitly constructed by using the basis elements of the symmetry algebra. Once this is achieved, the equation can be integrated by quadratures, as in the case of solvable symmetry algebras, although $\mathfrak{sl}(2,\mathbb{R})$ is not solvable.

A basis $\{ \campov_1, \campov_2, \campov_3 \} $ of the Lie symmetry algebra of equation (\ref{ode3}) verifying
\begin{gather}\label{co1}
[\campov_1,\campov_3]=\campov_1,\qquad [\campov_1,\campov_2]=2 \campov_3,\qquad [\campov_3,\campov_2]=\campov_2
\end{gather} can always be chosen \cite{accion}. Most of the approaches to integrate equations of the form (\ref{ode3}) that admit a Lie symmetry algebra isomorphic to $\mathfrak{sl}(2,\mathbb{R})$ \cite{olverclarkson,hydon2000symmetry,ibragimovnucci} use $\campov_1$ or $\campov_2$ to reduce~(\ref{ode3}) because any of them determines a two-dimensional algebra with $\campov_3$. The use of $\campov_3$ seems to be the worst choice to reduce the order of~(\ref{ode3}), because both~$\campov_1$ and~$\campov_2$ are lost as Lie point symmetries for the reduced equation, i.e., they are type I hidden symmetries~\cite{shrauner95,hidden1993}. Nevertheless, these basis elements can be recovered as $\mathcal{C}^{\infty}$-symmetries \cite{algebras_no_resolubles}, which, as it is shown in this section, will play an important role in the construction of the solvable structure.

If we choose the Lie point symmetry $\campov_3$ to reduce the order of equation (\ref{ode3}), then we can introduce canonical coordinates $(y,\alpha)$ for $\campov_3$, i.e., a~local change of variables on an open set $M$ of the variables $(x,u)$ of equation (\ref{ode3}),
\begin{gather*}\varphi(x,u)=(y(x,u),\alpha(x,u)),\end{gather*}
in which $\campov_3$ becomes $\partial_{\alpha}$.
Let $\alpha_1=\frac{d\alpha}{dy}$ be denoted by $w$ and let $w_i=\alpha_{i+1}$ for $1\leq i\leq 2$. Locally, equation (\ref{ode3}) can be written in terms of the invariants $\{y,w,w_1,w_2\}$ of $\campov_3^{(2)}$ as a reduced equation
\begin{gather}\label{ode_redu}
w_{2}=\widetilde{\phi}(y,w,w_1),
\end{gather}
def\/ined for $(y,w) \in M_1$ for some open set $M_1$. In this section $\campoA_{(x,u)}$ will denote the vector f\/ield associated to equation (\ref{ode3}), $\campoA_{(y,\alpha)}$ will be the vector f\/ield associated to equation $\alpha_3=\widetilde{\phi}(y,\alpha_1,\alpha_2)$, and $\campoA_{(y,w)}$ will be the vector f\/ield associated to equation~(\ref{ode_redu}).

The basis elements $\campov_1$ and $\campov_2$ are lost as Lie point symmetries for equation (\ref{ode_redu}), because~$\campov_1^{(1)}$ and~$\campov_2^{(1)}$ are not well-def\/ined vector f\/ields in the $(y,w)$-coordinates (they are exponential vector f\/ields~\cite{olver2000applications}). However, they can be recovered as $\mathcal{C}^{\infty}$-symmetries for equation~(\ref{ode_redu})~\cite{algebras_no_resolubles}. For that purpose we consider two nonzero functions $\varsigma_1,\varsigma_2\in\mathcal{C}^{\infty}(M)$ such that
\begin{gather}\label{varsigmas}
\campov_3(\varsigma_1)=\varsigma_1, \qquad \campov_3(\varsigma_2)=-\varsigma_2.
\end{gather} Observe that $\varsigma_2$ can be constructed from $\varsigma_1$ as $\varsigma_2=1/\varsigma_1$ and vice versa. With this choice we get \begin{gather} \label{conmutav3}
 \big[\campov_3^{(1)},\varsigma_1 \campov_1^{(1)}\big] =\big[\campov_3^{(1)},\varsigma_2 \campov_2^{(1)}\big]= 0,
 \end{gather} which can be checked by using (\ref{co1}) and the properties of the Lie bracket.
 The vector f\/ields~$\varsigma_1 \campov_1^{(1)}$ and~$\varsigma_2 \campov_2^{(1)}$ are projectable~\cite{sardanashvily2013advanced} with respect to the projection
 \begin{gather*}\begin{array}{@{}lrll}
 \pi_{\campov_3}^{(1)}\colon &\varphi^{(1)}\big(M^{(1)}\big) &\rightarrow & M_1,\\
 &(y,\alpha,w) &\mapsto & (y,w),
 \end{array}\end{gather*}
 because (\ref{conmutav3}) holds.
Let \begin{gather}\label{Y1Y2}
 \overline{\campov}_1= \big(\pi_{\campov_3}^{(1)}\big)_{*}\big(\varsigma_1 \campov_1^{(1)}\big), \qquad \overline{\campov}_2= \big(\pi_{\campov_3}^{(1)}\big)_{*}\big(\varsigma_2 \campov_2^{(1)}\big)
 \end{gather} denote the corresponding projected vector f\/ields. By Theorem~3 in~\cite{algebras_no_resolubles} the pairs
 $(\overline{\campov}_1, \lambda_1)$ and $(\overline{\campov}_2, \lambda_2)$
 are $\mathcal{C}^{\infty}$-symmetries of the equation~(\ref{ode_redu}) for the functions
 \begin{gather}\label{lambdas}
 \lambda_1=-\frac{\mathbf{A}_{(y,\alpha)}(\varsigma_1)}{\varsigma_1},\qquad \lambda_2=-\frac{\mathbf{A}_{(y,\alpha)}(\varsigma_2)}{\varsigma_2},
 \end{gather}
 respectively. In what follows we denote
 \begin{gather} \label{is}
 \campoY_1=\overline{\campov}_1^{[\lambda_1,(1)]}, \qquad \campoY_2=\overline{\campov}_2^{[\lambda_2,(1)]}.
 \end{gather}

For $i=1,2$, let $I_i=I_i(x,w,w_1)$ be a nonconstant f\/irst integral of $\campoA_{(y,w)}$ associated to the $\mathcal{C}^{\infty}$-symmetry $(\overline{\campov}_i,\lambda_i)$, i.e., $\campoY_i(I_i)=\campoA_{(y,w)}(I_i)=0$. The existence of such functions is warranted by Frobenius theorem, as it was discussed in Section~\ref{section2}. A moment of ref\/lection reveals that these two f\/irst integrals, written in terms of the original variables $(x,u,u_1,u_2)$, are also f\/irst integrals of the original third-order equation (\ref{ode3}) \cite{adriancedya2015, adrianlibro2016}. In fact,
\begin{gather}\label{AdeI}
\campoA_{(x,u)}(I_1)=\campov_1^{(2)}(I_1)=0,\qquad \campoA_{(x,u)}(I_2)=\campov_2^{(2)}(I_2)=0.
\end{gather} Since $I_1$ and $I_2$ can be written in terms of the dif\/ferential invariants of $\campov_3$ then
\begin{gather}\label{v3deI}
\campov_3^{(2)}(I_1)=\campov_3^{(2)}(I_2)=0.
\end{gather}

Previous discussion provides the following result:
\begin{Theorem}\label{teorema0}
Let $I_1=I_1(y,w,w_1)$ and $I_2=I_2(y,w,w_1)$ be two first integrals of $\mathbf{A}_{(y,w)}$ associated to the $\mathcal{C}^{\infty}$-symmetries $(\overline{\mathbf{v}}_1, \lambda_1)$ and $(\overline{\mathbf{v}}_2, \lambda_2)$ defined by \eqref{Y1Y2} and~\eqref{lambdas}. Then the functions
\begin{gather}\label{int_prim}
I_i=I_i(y(x,u),w(x,u,u_1),w_1(x,u,u_1,u_2)), \qquad i=1,2,
\end{gather}
are two functionally independent first integrals of $\mathbf{A}_{(x,u)}$ such that
\begin{gather}\label{condi1}
\mathbf{v}_1^{(2)}(I_1)=\mathbf{v}_3^{(2)}(I_1)=0,\qquad \mathbf{v}_2^{(2)}(I_2)=\mathbf{v}_3^{(2)}(I_2)=0.
\end{gather}
\end{Theorem}

Relations (\ref{AdeI}) and (\ref{v3deI}) imply that $\campov_2^{(2)}(I_1)\neq0$ and $\campov_1^{(2)}(I_2)\neq0$, because $\campoA_{(x,u)}$, $\campov_1^{(2)}$, $\campov_2^{(2)}$, and $\campov_3^{(2)}$ are independent vector f\/ields on the four-dimensional space of variables $(x,u,u_1,u_2)$;
hence the functions
\begin{gather}\label{F}
F_1=\frac{1}{\campov_1^{(2)}(I_2)},\qquad F_2=\frac{1}{\campov_2^{(2)}(I_1)}
\end{gather} are well def\/ined. Our aim is to prove that these functions can be used to construct a solvable structure with respect to $\{\campoA_{(x,u)}\}$ by using the basis elements of the Lie symmetry algebra. Before that, we need to establish some properties satisf\/ied by the functions $F_1$ and $F_2$ def\/ined in~(\ref{F}).

\begin{Lemma} \label{sistemas_3}
 The functions $F_1$ and $F_2$ defined by \eqref{F} satisfy
 \begin{alignat}{4}
& \mathbf{v}_3^{(2)}(F_1) = F_1,\qquad && \mathbf{A}_{(x,u)}(F_1) = 0,\qquad && \mathbf{v}_2^{(2)}(F_1) = 0;&\nonumber\\
& \mathbf{v}_3^{(2)}(F_2) = -F_2,\qquad && \mathbf{A}_{(x,u)}(F_2) = 0,\qquad && \mathbf{v}_1^{(2)}(F_2) = 0.& \label{tres}
 \end{alignat}
 Consequently, both systems in \eqref{tres} are compatible.
\end{Lemma}

\begin{proof}
The equality $\big[\campov_1^{(2)},\campov_3^{(2)}\big](I_2)=\campov_1^{(2)}(I_2)$, which comes from (\ref{co1}), yields
 \begin{gather*}
 \campov_3^{(2)}\bigl(\campov_1^{(2)}(I_2)\bigr)=-\campov_1^{(2)}(I_2),
 \end{gather*} because of (\ref{condi1}). Thus,
 \begin{gather*} \campov_3^{(2)}(F_1) =\campov_3^{(2)}\left(\frac{1}{\campov_1^{(2)}(I_2)}\right)=-\frac{\campov_3^{(2)}\big(\campov_1^{(2)}(I_2)\big)}
 {\bigl(\campov_1^{(2)}(I_2)\bigr)^2}=\frac{\campov_1^{(2)}(I_2)}{\bigl(\campov_1^{(2)}(I_2)\bigr)^2}=F_1.\end{gather*}
The relation $\campov_3^{(2)}(F_2)=-F_2$ can be deduced in a similar way, from $\big[\campov_3^{(2)},\campov_2^{(2)}\big](I_1)=\campov_2^{(2)}(I_1)$. Therefore
 \begin{gather} \label{resultado1}
 \campov_3^{(2)}(F_1)=F_1, \qquad \campov_3^{(2)}(F_2)=-F_2.
 \end{gather}
Since $\campov_i$ is a Lie point symmetry of (\ref{ode3}), then $ \big[\campov_i^{(2)},\campoA_{(x,u)}\big]\! =\! \rho_i \campoA_{(x,u)}$, where $\rho_i \!=\! -\campoA_{(x,u)}(\campov_i(x))$ for $i=1,2$. Therefore
 \begin{gather*}\campoA_{(x,u)}\bigl(\campov_1^{(2)}(I_2)\bigr)=- \big[\campov_1^{(2)},\campoA_{(x,u)}\big](I_2) =-\rho_1 \campoA_{(x,u)}(I_2)=0,\end{gather*} because $\campoA_{(x,u)}(I_2)=0$. Similarly, $\campoA_{(x,u)}\bigl(\campov_2^{(2)}(I_1)\bigr)=0$.
 Consequently,
 \begin{gather} \label{resultado2}
 \campoA_{(x,u)}(F_1) =\campoA_{(x,u)}\left(\frac{1}{\campov_1^{(2)}(I_2)}\right)=0, \qquad
 \campoA_{(x,u)}(F_2) =\campoA_{(x,u)}\left(\frac{1}{\campov_2^{(2)}(I_1)}\right)=0.
 \end{gather}
By using (\ref{co1}) we can write $\big[\campov_1^{(2)},\campov_2^{(2)}\big](I_1)=2\campov_3^{(2)}(I_1)$, which yields $\campov_1^{(2)}\bigl(\campov_2^{(2)}(I_1)\bigr)=0$, because of~(\ref{condi1}). By taking~$I_2$ instead of~$I_1$, the relation $\campov_2^{(2)}\bigl(\campov_1^{(2)}(I_2)\bigr)=0$ also holds. Therefore
 \begin{gather} \label{resultado3}
 \campov_1^{(2)}(F_1) =\campov_2^{(2)}\left(\frac{1}{\campov_1^{(2)}(I_2)}\right)=0, \qquad
 \campov_2^{(2)}(F_2) =\campov_1^{(2)}\left(\frac{1}{\campov_2^{(2)}(I_1)}\right)=0.
 \end{gather}

 Relations (\ref{resultado1}), (\ref{resultado2}) and (\ref{resultado3}) prove that the functions (\ref{F}) satisfy~(\ref{tres}).
\end{proof}

The existence of nontrivial solutions for systems (\ref{tres}) is the key to construct a solvable structure from the basis elements of the Lie symmetry algebra, as it is shown in the following theorem.

\begin{Theorem}\label{teorema_final}\sloppy
Let $F_1$ and $F_2$ be two functions satisfying \eqref{tres}. Then the ordered sets $\big\langle \mathbf{A}_{(x,u)},$ $\mathbf{v}_3^{(2)}, F_1 \mathbf{v}_1^{(2)} , F_2 \mathbf{v}_2^{(2)} \big\rangle$ and $\big\langle\mathbf{A}_{(x,u)}, \mathbf{v}_3^{(2)}, F_2 \mathbf{v}_2^{(2)} ,F_1 \mathbf{v}_1^{(2)} \big\rangle$ are solvable structures with respect to~$\{\mathbf{A}_{(x,u)}\}$.
\end{Theorem}

\begin{proof}
 Since for $i=1,2,3$, the vector f\/ield $\campov_i=\xi_i(x,u)\partial_x+\eta_i(x,u)\partial_u$ is a Lie point symmetry of~(\ref{ode3}) then
\begin{gather}\label{corcheteLie}
 \big[\campov_i^{(2)},\campoA_{(x,u)}\big] = \rho_i \campoA_{(x,u)},
\end{gather}
where $\rho_i = -\campoA_{(x,u)}(\xi_i)$. Obviously $\campov_3^{(2)}$ is a symmetry of $\{\campoA_{(x,u)}\}$, in the sense of Def\/inition~\ref{simetria}.
By using (\ref{corcheteLie}) and that $F_1$ and $F_2$ satisfy (\ref{tres}), the following commutation relations can be checked:
\begin{gather}
\big[F_1 \campov_1^{(2)},\campoA_{(x,u)}\big] = F_1 \rho_1 \campoA_{(x,u)},\qquad \big[F_2 \campov_2^{(2)},\campoA_{(x,u)}\big] = F_2 \rho_2 \campoA_{(x,u)},\nonumber\\
 \big[\campov_3^{(2)},F_1 \campov_1^{(2)}\big] = \big[\campov_3^{(2)},F_2 \campov_2^{(2)}\big]=0,\qquad
 \big[F_1 \campov_1^{(2)},F_2 \campov_2^{(2)}\big] = 2 F_1 F_2 \campov_3^{(2)}.\label{cuentas}
\end{gather}

According to Def\/inition \ref{simetria}, these relations prove that:
 \begin{enumerate}\itemsep=0pt
 \item $F_1 \campov_1^{(2)}$ and $F_2 \campov_2^{(2)}$ are symmetries of $\big\{\campoA_{(x,u)},\campov_3^{(2)}\big\}$.
 \item $F_2 \campov_2^{(2)}$ (resp.~$F_1 \campov_1^{(2)}$) is a symmetry of $\big\{\campoA_{(x,u)},\campov_3^{(2)},F_1 \campov_1^{(2)}\big\}$ (resp.~$\big\{\campoA_{(x,u)},\campov_3^{(2)},F_2 \campov_2^{(2)}\big\}$).
 \end{enumerate} The result follows from Def\/inition~\ref{simetria}.
\end{proof}

\section{Strategies for obtaining a complete system of f\/irst integrals}\label{section4}

The previous discussion shows that any pair of particular solutions $F_1$, $F_2$ of the respective system in (\ref{tres}) permits the construction of two solvable structures for a third-order equation with Lie symmetry algebra isomorphic to $\mathfrak{sl}(2,\mathbb{R})$. By Theorem~\ref{teorema_final}, such functions $F_1$ and $F_2$ provide the solvable structures $\big\langle\campoA_{(x,u)},\campov_3^{(2)},F_1\campov_1^{(2)},F_2\campov_2^{(2)}\big\rangle$ and $\big\langle\campoA_{(x,u)},\campov_3^{(2)},F_2\campov_2^{(2)},F_1\campov_1^{(2)}\big\rangle$ with respect to~$\{\campoA_{(x,u)}\}$. Therefore, the integrability by quadratures of the given ODE is warranted by Theorem~\ref{formas_exactas}. In this section we analyze three dif\/ferent strategies that can be followed to integrate completely the given equation.

 \textbf{Method 1:} Once two particular solutions $F_1$ and $F_2$ of (\ref{tres}) have been found, the method based on solvable structures \cite{basarab} (see also \cite{barko,athorne,prince}) can be applied to f\/ind by quadratures three independent f\/irst integrals of $\campoA_{(x,u)}$.
 Denote $\boldsymbol{\Omega} = dx \wedge du \wedge du_1 \wedge du_2$ and consider the corresponding dif\/ferential 1-forms (\ref{formas_b}) associated to the solvable structure $\big\langle\campoA,\campov_3^{(2)},F_1\campov_1^{(2)},F_2\campov_2^{(2)}\big\rangle$:
 \begin{gather}
 \boldsymbol{\omega}_3 = \frac{1}{F_2}\cdot\frac{\campov_1^{(2)} \intprod \campov_3^{(2)} \intprod \campoA_{(x,u)} \intprod \boldsymbol{\Omega}}{ \campov_{2}^{(2)} \intprod \campov_1^{(2)}\intprod \campov_3^{(2)} \intprod \campoA_{(x,u)} \intprod \boldsymbol{\Omega}},\nonumber\\
 \boldsymbol{\omega}_2= \frac{1}{F_1}\cdot\frac{\campov_2^{(2)} \intprod \campov_3^{(2)} \intprod \campoA_{(x,u)} \intprod \boldsymbol{\Omega}}{ \campov_{2}^{(2)} \intprod \campov_1^{(2)}\intprod \campov_3^{(2)} \intprod \campoA_{(x,u)} \intprod \boldsymbol{\Omega}},\nonumber\\
 \boldsymbol{\omega}_1 = \displaystyle\frac{\campov_2^{(2)} \intprod \campov_1^{(2)} \intprod \campoA_{(x,u)} \intprod \boldsymbol{\Omega}}{ \campov_{2}^{(2)} \intprod \campov_1^{(2)}\intprod \campov_3^{(2)} \intprod \campoA_{(x,u)} \intprod \boldsymbol{\Omega}}. \label{omegas}
\end{gather}

The 1-form $\boldsymbol{\omega}_3$ is (locally) exact and a function $\Theta_1$ such that
\begin{gather} \label{theta1}
 d\Theta_1 = \boldsymbol{\omega}_3
 \end{gather} is a common f\/irst integral to the system $\big\{ \campoA_{(x,u)}, \campov_3^{(2)}, F_1\campov_{1}^{(2)}\big\}$.

Since $\big\langle\campoA,\campov_3^{(2)},F_2\campov_2^{(2)},F_1\campov_2^{(2)}\big\rangle$ is also a solvable structure with respect to $\{\campoA_{(x,u)}\}$, the roles of $F_1\campov_1^{(2)}$ and $F_2\campov_2^{(2)}$ can be interchanged and thus $\boldsymbol{\omega}_2$ is also (locally) exact. A function $\Theta_2$ such that
\begin{gather} \label{theta2}
d \Theta_2 = \boldsymbol{\omega}_2
\end{gather}
is a common f\/irst integral to the system $\big\{ \campoA_{(x,u)}, \campov_3^{(2)},F_2\campov_2^{(2)}\big\}$.

Finally, $\boldsymbol{\omega}_1$ is exact modulo $\boldsymbol{\omega}_2$ and $\boldsymbol{\omega}_3$, i.e., $d \boldsymbol{\omega}_1\in \mathcal{I}\{\boldsymbol{\omega}_2,\boldsymbol{\omega}_3\}$, where $\mathcal{I}\{\boldsymbol{\omega}_2,\boldsymbol{\omega}_3\}$ stands for the ideal generated by $\boldsymbol{\omega}_2$ and $\boldsymbol{\omega}_3$ under taking exterior products. A function $\Theta_3$ such that
 \begin{gather} \label{theta3}
 d\Theta_3 = \boldsymbol{\omega}_1 \qquad \text{mod} \ \{ \boldsymbol{\omega}_2, \boldsymbol{\omega}_3\}
 \end{gather}
 completes the set $\{\Theta_1,\Theta_2,\Theta_3\}$ of independent f\/irst integrals of the vector f\/ield $\campoA_{(x,u)}$.

\textbf{Method 2:} We recall that the compatibility of systems (\ref{tres}) has been proved by construc\-ting the particular solutions given in~(\ref{F}):
\begin{gather*}
 F_1=\frac{1}{\campov_1^{(2)}(I_2)}, \qquad F_2=\frac{1}{\campov_2^{(2)}(I_1)}.
\end{gather*} By Theorem~\ref{teorema0} the functions $I_1$ and $I_2$ can be found through (\ref{int_prim}) from two known f\/irst integrals of the reduced equation (\ref{ode_redu}). Although these solutions $F_1$ and $F_2$ could be used to follow Method~1, the construction and integration of (\ref{omegas}) is not necessary: the functions~$I_1$,~$I_2$ given in~(\ref{int_prim}) and $F_1$, $F_2$ are themselves f\/irst integrals of~$\campoA_{(x,u)}$ because of~(\ref{AdeI}) and Lemma~\ref{sistemas_3}, respectively.

\looseness=-1
Let us prove that $\{I_1,F_1,F_2\}$ are functionally independent. If $F_1=\psi(I_1,F_2)$, for some function~$\psi$, then $\campov_1^{(2)}(F_1)=0$, because $\campov_1^{(2)}(F_2)=\campov_1^{(2)}(I_1)=0$. By (\ref{tres}), $\campov_2^{(2)}(F_1)=0$, and hence $[\campov_1^{(2)},\campov_2^{(2)}](F_1)=0$. Therefore (\ref{co1}) implies that $\campov_3^{(2)}(F_1)=0$, which cannot happen by~(\ref{tres}). A~similar reasoning proves that $\{I_2,F_1,F_2\}$ is also a complete set of f\/irst integrals of~$\campoA_{(x,u)}$.

\textbf{Method 3:} By Theorem \ref{teorema0} the function $I_1$ given in (\ref{int_prim}) is a common f\/irst integral to the set $\big\{ \campoA_{(x,u)}, \campov_3^{(2)}, \campov_{1}^{(2)} \big\}$; by (\ref{omegas}) and~(\ref{theta1}), $\Theta_1$ is a common f\/irst integral to the set $\big\{ \campoA_{(x,u)}, \campov_3^{(2)}, F_1\campov_{1}^{(2)} \big\}$. Since both sets of vector f\/ields are equivalent, $\Theta_1$ and $I_1$ are functionally dependent. Similarly, $\Theta_2$ and $I_2$ must be functionally dependent. Therefore, if $\Theta_3$ satisf\/ies (\ref{theta3}), then the set $\{I_1,I_2,\Theta_3\}$ is a complete system of f\/irst integrals of~$\campoA_{(x,u)}$.

In consequence, provided the f\/irst integrals~(\ref{int_prim}) given in Theorem~\ref{teorema0} are known, the complete solution of (\ref{ode3}) arises by f\/inding by quadrature a primitive $\Theta_3$ of $\boldsymbol{\omega}_1$, restricted to $I_1={C_1}$, $I_2={C_2}$, where $C_1,C_2\in \mathbb{R}$. We point out that~$\boldsymbol{\omega}_1$ can be directly computed from the basis elements $\{\campov_1,\campov_2,\campov_3\}$.

The second and the third of the described methods use the functions~(\ref{int_prim}), obtained from two f\/irst integrals of the reduced equation (\ref{ode_redu}), to perform the complete integration of equation~(\ref{ode3}). The procedure described in Section \ref{subsection1} can be applied to integrate by quadratures the reduced equation~(\ref{ode_redu}), because this equation admits the pairs $(\overline{\campov}_1, \lambda_1)$, $(\overline{\campov}_2, \lambda_2)$ def\/ined by~(\ref{Y1Y2}) and~(\ref{lambdas}) as $\mathcal{C}^{\infty}$-symmetries. That procedure works provided that two particular solutions~$\overline{h}_1$ and $\overline{h}_2$ of the corresponding system in~(\ref{hs}) are known (Theorem~\ref{integrabilidadorden2}).

A remarkable fact is that the particular solutions $F_1$, $F_2$ of (\ref{tres}) given in (\ref{F}) can be directly expressed in terms of $\overline{h}_1$ and $\overline{h}_2$, without the need to compute $I_1$ and $I_2$. In fact, functions (\ref{F}) can be written in the form
\begin{gather}\label{Fconhbarra}
 F_1= \frac{1}{\campov_1^{(2)}(I_2)}=\frac{\varsigma_1}{\campoY_1(I_2)}=\frac{\varsigma_1\overline{h}_1}{Q_1},
 \qquad F_2=\frac{1}{\campov_2^{(2)}(I_1)}=\frac{\varsigma_2}{\campoY_2(I_1)}=\frac{\varsigma_2\overline{h}_2}{Q_2},
\end{gather}
 where $\varsigma_1$ and $\varsigma_2$ are given by (\ref{varsigmas}) and $Q_1$, $Q_2$ are the respective characteristics of the vector f\/ields $\overline{\campov}_1$ and $\overline{\campov}_2$ given in~(\ref{Y1Y2}). Expressions (\ref{Fconhbarra}) are easy to check by taking~(\ref{Y1Y2}),~(\ref{v3deI}), and~(\ref{hconX}) into account. All the functions involved in~(\ref{Fconhbarra}) are assumed to be written in terms of the original variables $(x,u,u_1,u_2)$.

The described methods to integrate by quadratures equation (\ref{ode3}) can be signif\/icantly simplif\/ied when two particular solutions $\overline{h}_1$ and $\overline{h}_2$ of the corresponding system in~(\ref{hs}) are known. The next alternatives can be followed:

\textbf{Option 1:} Use the function $\overline{h}_2$ to construct by quadrature a primitive $I_1$ of the 1-form $\boldsymbol{\beta}_1$ def\/ined in (\ref{omegas0}). Once written in variables $(x,u,u_1,u_2)$, the function~$I_1$, and the functions~$F_1$,~$F_2$ given in~(\ref{Fconhbarra}) are three functionally independent f\/irst integrals associated to the original third-order equation. Alternatively, the function~$\overline{h}_1$ can be used to compute by quadrature a f\/irst integral $I_2$ from a primitive of $ \boldsymbol{\beta}_2$ def\/ined in~(\ref{omegas0}) and to construct the complete set of f\/irst integrals $\{I_2,F_1,F_2\}$. For this option, only one quadrature (to compute either~$I_1$ or~$I_2$) is required.

\textbf{Option 2:} Use the functions $\overline{h}_1$ and $\overline{h}_2$ to construct by means of two quadratures both f\/irst integrals $I_1$ and $I_2$ as primitives of~(\ref{omegas0}). Finally, f\/ind by quadrature a~primitive~$\Theta_3$ of~$\boldsymbol{\omega}_1$, restricted to $I_1={C_1}$, $I_2={C_2}$, where $C_1,C_2\in \mathbb{R}$.

\textbf{Option 3:} Use the functions $\overline{h}_1$ and $\overline{h}_2$ to construct the functions $F_1$ and $F_2$ given in~(\ref{Fconhbarra}) and follow the Method~1. For this option, three successive quadratures to f\/ind the primitives $\Theta_1$, $\Theta_2$ and $\Theta_3$ satisfying~(\ref{theta1}), (\ref{theta2}), and (\ref{theta3}), are necessary.

In the following sections this procedure is applied to integrate by quadratures two third-order ODEs admitting Lie symmetry algebras that are isomorphic to $\mathfrak{sl}(2,\mathbb{R})$. The corresponding f\/irst integrals and general solutions can be expressed in terms of two independent solutions of second-order \textit{linear} equations.

\section{Example I}\label{section6}

The third-order ordinary dif\/ferential equation
\begin{gather}\label{ejemplo}
2 u_1 u_3 - 3 u_2^2 - 2 u u_1^4 =0.
\end{gather}
admits the following Lie point symmetries
\begin{gather}\label{sim}
\campov_1 = \partial_x, \qquad \campov_2 = x^2 \partial_x, \qquad \campov_3 = x \partial_x,
\end{gather}
which satisfy relations (\ref{co1}) and generate a Lie symmetry algebra isomorphic to $\mathfrak{sl}(2,\mathbb{R})$. Our aim is to f\/ind solutions $F_1$ and $F_2$ of (\ref{tres}) to construct the solvable structures warranted by Theorem~\ref{teorema_final}. Once this is achieved, equation~(\ref{ejemplo}) can be integrated by quadratures by following any of the strategies described in Section~\ref{section4}.

The direct search of particular solutions $F_1$, $F_2$ of (\ref{tres}) seems not to be an easy task. We use the results given in Section \ref{subsection1} to construct $F_1$, $F_2$ as in (\ref{Fconhbarra}) by using solutions of the corresponding systems~(\ref{hs}).

{\bf Solutions of systems (\ref{hs}).} Systems (\ref{hs}) refer to the second-order equation obtained from (\ref{ejemplo}) by using the Lie point symmetry~$\mathbf{v}_3$. Such reduced equation becomes
\begin{gather} \label{ec_red}
w_2 = \frac{3w_1^2}{2w}+\frac{1}{2} w^3 -yw,
\end{gather} by using the invariants $y=u$ and $w=\frac{1}{x u_1}$ of $\mathbf{v}_3^{(1)}$.

The Lie point symmetries $\campov_1$ and $\campov_2$ can be recovered as $\mathcal{C}^{\infty}$-symmetries of (\ref{ec_red}) by using two functions $\varsigma_1$, $\varsigma_2$ satisfying (\ref{varsigmas}), which can be easily calculated
\begin{gather}\label{proj}
\varsigma_1(x,u)=x, \qquad \varsigma_2(x,u)=\frac{1}{x}.
\end{gather}
According to (\ref{Y1Y2}) and (\ref{lambdas}), these inherited $\mathcal{C}^{\infty}$-symmetries are def\/ined by the pairs $(- w \partial_w, w)$ and $(w \partial_w,w)$, respectively. By using the characteristics
\begin{gather}\label{Qejemplo}
Q_1(y,w,w_1)=-w, \qquad Q_2(y,w,w_1)=w,
\end{gather}
it can be checked that
the f\/irst-order $\lambda$-prolongations of their respective canonical representatives (\ref{canonico}) become
\begin{gather*}\campoX_1 = \partial_w + \left(-w+\frac{w_1}{w} \right) \partial_{w_1}, \qquad \campoX_2 = \partial_w + \left(w+\frac{w_1}{w} \right)\partial_{w_1}.
\end{gather*}
Systems (\ref{hs}) can be constructed with these vector f\/ields and the vector f\/ield $\campoA_{(y,w)}$ associated to equation~(\ref{ec_red}).
Two of their particular solutions ${\overline{h}_1}={\overline{h}_1}(y,w,w_1) $ and ${\overline{h}_2}={\overline{h}_2}(y,w,w_1)$ are given by
\begin{gather}\label{hbarejemplo}
{\overline{h}_1} = -\frac{\left(({w^2-w_1}) \Psi_2(y) - {2w} \Psi_2'(y)\right)^2 }{{4w^2} W(\Psi_1,\Psi_2)(y)},\qquad
{\overline{h}_2} =\frac{\left(({w^2+w_1}) \Psi_2(y) + {2w} \Psi_2'(y) \right)^2} {{4w^2} W(\Psi_1,\Psi_2)(y)},
\end{gather}
where $\Psi_1=\Psi_1(y)$ and $\Psi_2=\Psi_2(y)$ are two independent solutions of the Airy equation
\begin{gather}\label{airy}
\Psi_{yy} = \frac{1}{2} y \Psi,
\end{gather} and $W(\Psi_1,\Psi_2)(y)=\Psi_1(y) \Psi_2'(y) - \Psi_1'(y) \Psi_2(y)$ denotes the corresponding Wronskian.

{\bf Solutions of systems (\ref{tres}) and solvable structures for equation (\ref{ejemplo}).} Once the functions (\ref{hbarejemplo}) have been determined, the two particular solutions (\ref{Fconhbarra}) for systems (\ref{tres}) can be determined without any additional integration. Expressions (\ref{Fconhbarra}) use the functions (\ref{proj}), (\ref{Qejemplo}), and (\ref{hbarejemplo}), written in the original variables $(x,u,u_1,u_2)$:
\begin{gather}
F_1(x,u,u_1,u_2)=\frac{\left({ \left({xu_2+2u_1} \right)\Psi_1(u) - {2x u_1^2} \Psi_1'(u) }\right)^2}{4 u_1^3 W(\Psi_1,\Psi_2)(u)} , \nonumber\\
 F_2(x,u,u_1,u_2)=\frac{\left({u_2\Psi_1(u)-2u_1^2\Psi_1'(u)}\right)^2}{4u_1^3 W(\Psi_1,\Psi_2)(u)}.\label{Feje}
\end{gather}
By Theorem \ref{teorema_final}, functions (\ref{Feje}) permit the construction of the solvable structures $\big\langle \campoA_{(x,u)}$, $\campov_3^{(2)}, F_1 \campov_1^{(2)}, F_2 \campov_2^{(2)} \big\rangle$ and $\big\langle \campoA_{(x,u)}, \campov_3^{(2)}, F_1 \campov_1^{(2)}, F_2 \campov_2^{(2)} \big\rangle$ with respect to the vector f\/ield $\campoA_{(x,u)}$ associated to equation (\ref{ejemplo}).

{\bf Complete sets of f\/irst integrals of equation (\ref{ejemplo}) obtained by quadratures and general solution.} Functions (\ref{hbarejemplo}) are all what we need to complete the solution by quadratures by following any of the alternatives enumerated in Section~\ref{section4}.

{\bf Option 1:} Once the functions (\ref{hbarejemplo}) have been determined, two independent f\/irst integrals for equation (\ref{ec_red}) can be calculated by quadratures as primitives of the 1-forms (\ref{omegas0}). Such functions written in variables $(x,u,u_1,u_2)$ provide two functionally independent f\/irst integrals of equation~(\ref{ejemplo}), which can be expressed in terms of the independent solutions $\Psi_1$ and $\Psi_2$ of the Airy equation~(\ref{airy}):
\begin{gather}
I_1(x,u,u_1,u_2) = \frac{u_2\Psi_1(u)-2u_1^2\Psi_1'(u)}{u_2\Psi_2(u)-2u_1^2\Psi_2'(u)}\nonumber,\\
I_2(x,u,u_1,u_2) = \frac{(xu_2+2u_1)\Psi_1(u) - {2x u_1^2} \Psi_1'(u) }{( xu_2+2u_1)\Psi_2(u) - {2x u_1^2} \Psi_2'(u)}.\label{i1}
\end{gather}

Any of the sets $\{I_1,F_1,F_2\}$ or $\{I_2,F_1,F_2\}$ is a complete system of functionally independent f\/irst integrals for the equation~(\ref{ejemplo}).

{\bf Option 2:} In this case the corresponding 1-form $\boldsymbol{\omega}_1$ in (\ref{omegas}) becomes
\begin{gather*}\boldsymbol{\omega}_1 = \frac{ 2 u_1 u_2 + x u_2^2 - 2 x u u_1^4}{2 u_1^3} du - \frac{u_1+2 u_2 x}{u_1^2} du_1 + \frac{x}{u_1} du_2.\end{gather*}
The restriction of $\boldsymbol{\omega}_1$ to the submanifold def\/ined by $I_1 = c_1$ and $I_2 = c_2$, ($c_1,c_2\in \mathbb{R}$), is exact and
\begin{gather}\label{theta3ejemplo}
\Theta_3(x,u,u_1,u_2) = \ln(x) + \ln\left(\frac{\Psi_1(u)-I_1{\Psi_2}(u)}{\Psi_1(u)-I_2{\Psi_2}(u)}\right)
\end{gather} is a primitive. As it was discussed in Section \ref{section4}, the set $\{I_1,I_2,\Theta_3\}$ def\/ined by the functions~(\ref{i1}) and~(\ref{theta3ejemplo}) is a complete system of f\/irst integrals for equation~(\ref{ejemplo}).

\textbf{Option 3:} The functions $F_1$ and $F_2$ given in (\ref{Feje}) can be used to construct the 1-forms~(\ref{omegas}), which can be successively integrated by quadratures. We omit the expressions for the f\/irst integrals $\{\Theta_1,\Theta_2\}$ corresponding to~(\ref{theta1}) and~(\ref{theta2}), because they are functionally dependent of the functions~$I_1$,~$I_2$ given in~(\ref{i1}). The computation of a remaining f\/irst integral $\Theta_3$ could be achieved as in Option~2.

The general solution of equation (\ref{ejemplo}) can be obtained, for instance, by setting $I_1=c_1$, $I_2=c_2$, $\Theta_3=\ln(c_3)$ and becomes
\begin{gather}\label{solucion1}
x= c_3 \frac{\Psi_1(u)-c_2{\Psi_2}(u)}{\Psi_1(u)-c_1{\Psi_2}(u)},
\end{gather} where $c_i\in\mathbb{R}$ for $i=1,2,3$, $c_1\neq c_2$, and $\Psi_1$ and $\Psi_2$ are two independent solutions of the Airy equation~(\ref{airy}).

\begin{Remark} Equation (\ref{ejemplo}) has been chosen for purposes of illustration of the procedures presented in this paper; it could have been solved by using other methods that appear in the literature.

{\bf 1.} By following \cite{ibragimovnucci}, equation (\ref{ejemplo}) can be reduced to a f\/irst-order ODE by using the dif\/ferential invariants
\begin{gather*}
y=u,\qquad m=\frac{u_2}{u_1^2}
\end{gather*}
of the two-dimensional subalgebra $\mathcal{L}_2$ generated by $\mathbf{v}_1$ and $\mathbf{v}_3$. The reduced f\/irst-order equation is the Riccati equation
\begin{gather}\label{ricatinuci}
m_1=-\frac{1}{2}m^2+y,
\end{gather}
which becomes the Airy equation (\ref{airy}) by means of the standard transformation $m(y)=\frac{\Psi'(y)}{2\Psi(y)}$.
 In consequence, the general solution of~(\ref{ricatinuci}) can be expressed in terms of two independent solutions, $\Psi_1=\Psi_1(y)$ and $\Psi_2=\Psi_2(y)$, of equation~(\ref{airy}).
Such general solution, written in terms of the original variables $(x,u,u_1,u_2)$, yields the second-order equation
\begin{gather}\label{segundoorden}
\frac{u_2}{u_1^2}=\frac{C\Psi_1'(u)+\Psi_2'(u)}{2(C\Psi_1(u)+\Psi_2(u))},
\end{gather} where $C \in \mathbb{R}$. Although~(\ref{segundoorden}) admits $\mathcal{L}_2$, and may therefore be integrated by quadratures, the expression obtained for its general solution
\begin{gather*}
x=C_1+C_2\int \frac{du}{(C\Psi_1(u)+\Psi_2(u))^2},
\end{gather*} requires a primitive which apparently cannot be easily evaluated. The procedure presented in this paper overcomes this dif\/f\/iculty, because provides expression (\ref{solucion1}) for the general solution of~(\ref{ejemplo}). This general solution could be also reached by following the procedure described in~\cite{olverclarkson}, see also~\cite{hydon2000symmetry}.

{\bf 2.} Apart from (\ref{sim}), equation (\ref{ejemplo}) admits three additional Lie point symmetries, whose inf\/initesimals can be expressed in terms of solutions of the Airy equation~(\ref{airy}) as follows
\begin{gather}\label{simairy}
\Psi_1(u)^2\partial_u,\qquad \Psi_2(u)^2\partial_u,\qquad \Psi_1(u)\Psi_2(u)\partial_u.
\end{gather} In (\ref{simairy}) $\Psi_1=\Psi_1(u)$ and $\Psi_2=\Psi_2(u)$ denote two linearly independent solutions of the equation $\Psi''(u) = \frac{1}{2} u \Psi(u)$.

Indeed, it can be checked that equation~(\ref{ejemplo}) satisf\/ies the conditions that appear in \cite[p.~235]{schwarz2007algorithmic}, although the corresponding solution algorithm could present some dif\/f\/iculties because a rational solution of equation (\ref{ricatinuci}) seems to be required.

Alternatively, equation (\ref{ejemplo}) can be transformed by a point transformation into the representative third-order equation with six-dimensional stabilizer (see \cite[Section~IV, Case~B.4]{gatomri}). The transformation can be found by using basis elements of the form (\ref{simairy}) such that $\Psi_1=\Psi_1(u)$ and $\Psi_2=\Psi_2(u)$ verify $W(\Psi_1,\Psi_2)(u)=1$. It can be checked that by introducing the new dependent variable $z=\frac{\Psi_2(u)}{\Psi_1(u)}$ equation~(\ref{ejemplo}) becomes
\begin{gather*}
z_3=\frac{3z_2^2}{2z_1}.
\end{gather*} This equation can be easily integrated by quadratures and its general solution $z(x)=\frac{k_1}{k_2+x}+k_3$ provides an alternative expression for the general solution of equation~(\ref{ejemplo}) obtained in~(\ref{solucion1}).
\end{Remark}

\section{Example II}\label{section7}

For the third-order equation
\begin{gather}\label{ejemplo2}
u_3 \big(u_1^2 - 2 u u_2\big) u^2 +1= 0,
\end{gather}
 the determining equations for the inf\/initesimals $\xi=\xi(x,u)$ and $\eta=\eta(x,u)$ of a Lie point symmetry reduce to
 \begin{gather}\label{ecdet}
 \eta=u\xi_u,\qquad \xi_{xxx}=0, \qquad \xi_u=0.
 \end{gather}
 If follows that the Lie invariance algebra of equation (\ref{ejemplo2}) is three-dimensional and it is generated by
\begin{gather}\label{generadores}
\campov_1 = \partial_x, \qquad \campov_2 = x^2 \partial_x+2 x u \partial_u, \qquad \campov_3=x \partial_x + u \partial_u.
\end{gather}
By (\ref{ecdet}) any Lie point symmetry of (\ref{ejemplo2}) must be a linear combination of~(\ref{generadores}). Vector f\/ields~(\ref{generadores}) satisfy the commutation relations~(\ref{co1}) and generate the nonsolvable Lie algebra $\mathfrak{sl}(2,\mathbb{R})$.

If we follow, for instance, the Option 2 in Section \ref{section4} to integrate equation (\ref{ejemplo2}), the procedure provides three independent f\/irst integrals that can be expressed in terms of two independent solutions, $\psi_1=\psi_1(z)$, $\psi_2=\psi_2(z)$, of the linear equation
 \begin{gather}\label{edo_lineal}
 z \psi''(z)-\psi'(z)-4 z^4 \psi(z)=0.
 \end{gather}

The following functions correspond to the f\/irst integrals (\ref{int_prim}) given in Theorem~\ref{teorema0}:
\begin{gather}
 I_1(x,u,u_1,u_2) = \frac{2 (-2 u+u_1x)(u_1^2 - 2 u u_2)\psi_1(u_1^2 - 2 u u_2)+x \psi_1'(u_1^2 - 2 u u_2)}{2 (-2 u+u_1x)(u_1^2 - 2 u u_2)\psi_2(u_1^2 - 2 u u_2)+x \psi_2'(u_1^2 - 2 u u_2)}, \nonumber\\
 I_2(x,u,u_1,u_2)=\frac{2 u_1 (u_1^2 - 2 u u_2) \psi_1(u_1^2 - 2 u u_2)+\psi_1'(u_1^2 - 2 u u_2) }{2 u_1 (u_1^2 - 2 u u_2) \psi_2(u_1^2-2uu_2)+\psi_2'(u_1^2 - 2 u u_2)}.\label{int_prim_ej_2}
\end{gather}

It can be checked that a primitive of the corresponding 1-form $\boldsymbol{\omega_1}$ def\/ined in (\ref{omegas}), restricted to the submanifold def\/ined by $I_1=C_1$ and $I_2=C_2$, is given by
 \begin{gather*}\Theta_3(x,u,u_1,u_2)= \ln \left(\ \frac{C_1 \psi_2(u_1^2 - 2 u u_2)-\psi_1(u_1^2 - 2 u u_2)}{x\bigl(C_2 \psi_2(u_1^2 - 2 u u_2)-\psi_1(u_1^2 - 2 u u_2)\bigr)} \right).\end{gather*}
After replacing $C_1$ and $C_2$ by the functions $I_1$ and $I_2$ given in (\ref{int_prim_ej_2}), the function $I_3=\exp(\Theta_3)$ becomes
 \begin{gather} \label{int_3_ej_2}
 I_3(x,u,u_1,u_2)= \frac{2 u_1 (u_1^2 - 2 u u_2)\psi_2(u_1^2 - 2 u u_2)+\psi_2'(u_1^2 - 2 u u_2)}{2(u_1^2 - 2 u u_2)(u_1x-2u) \psi_2(u_1^2 - 2 u u_2)+x\psi_2'(u_1^2 - 2 u u_2)}.
 \end{gather}
The functions given in (\ref{int_prim_ej_2}) and (\ref{int_3_ej_2}) are three independent f\/irst integrals for equation~(\ref{ejemplo2}); they provide its implicit solution \begin{gather}
I_1(x,u,u_1,u_2)=C_1,\qquad I_2(x,u,u_1,u_2)=C_2, \qquad I_3(x,u,u_1,u_2)=C_3,\label{impsol}
\end{gather}
where $C_i\in \mathbb{R}$  for $i=1,2,3$. The elimination of $u_1$ and $u_2$ from~(\ref{impsol}) to obtain an explicit solution of (\ref{ode3}) seems practically impossible: the functions $\psi_1$ and $\psi_2$ and their derivatives in~(\ref{impsol}) are evaluated on $u_1^2 - 2 u u_2$. Our aim is to obtain a parametric solution for equa\-tion~(\ref{ejemplo2}). For that purpose, we consider any solution $\psi=\psi(z)$ of~(\ref{edo_lineal}) and introduce a parameter $t>0$ such that $z(t)= \sqrt{2 t}$. With this choice, the function $\phi(t)=\psi(z(t))$ is a solution of the Schr\"odinger equation
 \begin{gather} \label{schro1}
 \phi''(t) - 4 \sqrt{2t} \phi(t) =0.
 \end{gather}
Conversely, if $\phi=\phi(t)$ is a solution of (\ref{schro1}), then $\psi(z)=\phi\big(\frac{z^2}{2}\big)$ satisf\/ies $\psi'(z)=z\phi'\big(\frac{z^2}{2}\big)$ and $\psi=\psi(z)$ is a solution of (\ref{edo_lineal}). The implicit solution~(\ref{impsol}) expressed in terms of the solutions~$\phi_1$ and~$\phi_2$ of~(\ref{schro1}) associated to~$\psi_1$ and~$\psi_2$, respectively, becomes
 \begin{gather}
 \frac{2(-2 u + u_1 x) \phi_1(s)+ x \phi_1'(s)}{2(-2 u+ u_1 x) \phi_2(s)+x \phi_2'(s)} = C_1,\qquad
 \frac{2u_1 \phi_1(s)+\phi_1'(s)}{2u_1 \phi_2(s)+ \phi_2'(s)} = C_2,\nonumber\\
 \frac{2 u_1 \phi_2(s)+ \phi_2'(s)}{2 (u_1 x - 2 u) \phi_2(s)+x \phi_2'(s)} = C_3,\label{im22}
 \end{gather}
 where
$s= \frac{(u_1^2 - 2 u u_2)^2}{2}$.
 From (\ref{im22}), we obtain the parametric solution for equation~(\ref{ejemplo2})
 \begin{gather} \label{gen2}
 x(s) = \frac{\phi_1(s) - C_1 \phi_2(s)}{C_3 (\phi_1(s)-C_2 \phi_2(s))}, \qquad
 u(s) = \frac{(C_2-C_1) W(\phi_1,\phi_2)(s)}{4 C_3 (\phi_1(s)-C_2 \phi_2(s))^2},
 \end{gather} where $\phi_1$ and $\phi_2$ are two independent solutions of the Schr\"odinger equation~(\ref{schro1}) and $W(\phi_1,\phi_2)$ stands for the corresponding Wronskian.

\begin{Remark}
The presence of a Schr\"odinger equation in the solution of a third-order equation with $\mathfrak{sl}(2,\mathbb{R})$ is not new in the literature: it appears in \cite{olverclarkson} (see also~\cite{hydon2000symmetry}) by using the relations by prolongation of the three possible realizations of $\mathfrak{sl}(2,\mathbb{C})$ on the $(x,u)$ plane. Such relations could have been also used to complete the solution of equation (\ref{ejemplo2}). It should be remarked that in our method no previous transformation is needed to map the basis elements into one of the canonical realizations of $\mathfrak{sl}(2,\mathbb{C})$.

Reductions of equation (\ref{ejemplo2}) to Riccati equations or to the associated second-order linear equations can be also obtained by the original Lie theory or by following the procedure in~\cite{ibragimovnucci}. For instance, it can be checked that equation (\ref{ejemplo2}) reduces to the Riccati equation
 \begin{gather} \label{ricatinuci2}
 m_1=2y^2m^2-y,
 \end{gather} where $m=\frac{1}{u_1}$ and $y=\frac{1}{2}u_1^2-u u_2$. Equation~(\ref{ricatinuci2})
 becomes the linear equation $y\varphi''(y)-2\varphi'(y)-2 y^4 \varphi(y)=0$
 by means of the standard transformation $m(y)=-\frac{\varphi'(y)}{2 y^2 \varphi(y)}$. By using two independent solutions $\varphi_1=\varphi_1(y)$ and $\varphi_2=\varphi_2(y)$ of that linear equation, the general solution of~(\ref{ricatinuci2}), written in terms of the original variables $(x,u,u_1,u_2)$, yields the second-order equation
 \begin{gather} \label{segundoorden2}
 u_1=-2\left(\frac{1}{2}u_1^2-u u_2\right)^2 \frac{C\varphi_2\left(\frac{1}{2}u_1^2-u u_2\right)-\varphi_1\left(\frac{1}{2}u_1^2-u u_2\right)}{C\varphi_2'\left(\frac{1}{2}u_1^2-u u_2\right)-\varphi_1'\left(\frac{1}{2}u_1^2-u u_2\right)}.
 \end{gather}
Equation (\ref{segundoorden2}) needs to be integrated in order to recover the general solution of (\ref{ejemplo2}). This can be avoided by following the method presented in this paper, which provides the parametric solution~(\ref{gen2}) in terms of a slight dif\/ferent second-order linear ODE (equation~(\ref{schro1})).
\end{Remark}

 \section{Concluding remarks and further extensions}

\looseness=-1 The well-known method to integrate by quadratures an ODE with a solvable symmetry algebra is not available for third-order ODEs admitting a Lie symmetry algebra that is isomorphic to the nonsolvable symmetry algebra $\mathfrak{sl}(2,\mathbb{R})$. Basis elements $\{\campov_1,\campov_2,\campov_3\}$ of the Lie symmetry algebra satisfying the relations (\ref{co1}) are used to construct explicitly solvable structures with respect to the vector f\/ield $\campoA$ associated to the ODE. In consequence, the given ODE can be integrated by quadratures.

Such solvable structures are of the form $\big\langle \campoA,\campov_3^{(2)},F_1\campov_1^{(2)},F_2\campov_2^{(2)}\big\rangle$ and $\big\langle \campoA,\campov_3^{(2)},F_2\campov_2^{(2)},F_1\campov_1^{(2)}\big\rangle $, where $F_1$ and $F_2$ are solutions of systems~(\ref{tres}). The existence of such solutions can be proved by a constructive procedure, which involves a second-order ODE obtained by reducing the order of the original ODE with~$\campov_3$. Such reduced equation admits two $\mathcal{C}^{\infty}$-symmetries inherited from~$\campov_1$ and~$\campov_2$. Two known f\/irst integrals of the reduced equation associated to the $\mathcal{C}^{\infty}$-symmetries can be used to construct the functions $F_1$ and $F_2$ (see~(\ref{F})). It is noteworthy that such f\/irst integrals can be determined by quadratures when two particular solutions $\overline{h}_1$, $\overline{h}_2$ of systems~(\ref{hs}) are known (Theorem~\ref{integrabilidadorden2}). In fact, these solutions $\overline{h}_1$, $\overline{h}_2$ lead directly to the construction of the functions~$F_1$ and~$F_2$ given in (\ref{Fconhbarra}), because the functions $\varsigma_1$, $\varsigma_2$, $Q_1$, and $Q_2$ are known from the basis elements of the Lie symmetry algebra.
Once these functions $F_1$ and $F_2$ are known, the given ODE can be integrated by quadratures (as in the case of solvable symmetry algebras) by using any of the three dif\/ferent strategies described in Section \ref{section4}.

In the presented procedure it is not necessary to map the basis elements of the Lie symmetry algebra into one of the four inequivalent canonical realizations of $\mathfrak{sl}(2,\mathbb{R})$ on the plane. For each realization of $\mathfrak{sl}(2,\mathbb{R})$ the corresponding class of third-order $\mathfrak{sl}(2,\mathbb{R})$-invariant ODEs involves arbitrary elements/functions of single arguments. The method has been illustrated with two particular examples that correspond to the f\/irst and second realizations of $\mathfrak{sl}(2,\mathbb{R})$, respectively. In both cases the general solution can be expressed in terms of two linearly independent solutions of some second-order linear equations, as in previous methods in the literature \cite{olverclarkson,ibragimovnucci} (see also~\cite{hydon2000symmetry}). Further applications of the method to arbitrary third-order $\mathfrak{sl}(2,\mathbb{R})$-invariant equations, including the study of the two remaining inequivalent realizations of~$\mathfrak{sl}(2,\mathbb{R})$ on the plane not considered in this paper, are currently being investigated. In particular, it remains the question whether the functions $F_1$ and $F_2$ can be always determined in terms of solutions of some second-order linear ODEs, as occurs in the particular examples presented in this paper. These extensions will be addressed in a separate work.

It is expected that the methods developed in this work can be adapted to integrate by quadratures equations with other nonsolvable symmetry algebras. A work in this line is also currently in progress.

\subsection*{Acknowledgements}

The constructive comments and the ef\/forts of the editor and referees to improve the contents of this paper are gratefully acknowledged. The authors also thank Professor J.L.~Romero for his assistance, patience and always valuable suggestions. This research was partially supported by the University of C\'adiz and Junta de Andaluc\'ia research group FQM~377. A.~Ruiz acknowledges the support of a grant of the University of C\'adiz program ``Movilidad Internacional, Becas UCA-Internacional-Posgrado'' during his stay at the University of Minnesota.

\pdfbookmark[1]{References}{ref}
\LastPageEnding

\end{document}